\documentclass[11pt,a4paper]{article}

\usepackage{epsf,epsfig,amsfonts,amsgen,amsmath,amstext,amsbsy,amsopn,amsthm
}
\usepackage{amsmath,times,mathptmx}
\usepackage{amsfonts,amsthm,amssymb}
\usepackage{amsfonts}
\usepackage{graphics}
\usepackage{latexsym,bm}
\usepackage{amsfonts,amsthm,amssymb,bbding}
\usepackage{indentfirst}
\usepackage{graphicx}
\usepackage{color}
\usepackage[colorlinks=true,anchorcolor=blue,filecolor=blue,linkcolor=blue,urlcolor=blue,citecolor=blue]{hyperref}
\usepackage{float}
\usepackage{tikz}
\evensidemargin=\oddsidemargin\topmargin=-1.5cm

\newtheorem{thm}{Theorem}[section]

\newtheorem{ques}{Question}[section]
\newtheorem{lem}{Lemma}[section]
\newtheorem{cor}{Corollary}[section]

\newtheorem{remark}{Remark}[section]

\newtheorem{conj}{Conjecture}[section]
\newtheorem{claim}{Claim}[section]
\newtheorem{definition}{Definition}[section]

\addtocounter{section}{0}

\begin{document}
\title{Spectral extremal results on edge blow-up of graphs\footnote{Supported by the National Natural Science Foundation of China (Nos.\,12271162,\,12171066), and Natural Science Foundation of Shanghai (No. 22ZR1416300).}}
\author{ {\bf Longfei Fang$^{a,b}$},
{\bf Huiqiu Lin$^{a}$}\thanks{Corresponding author: huiqiulin@126.com(H. Lin)}
\\
\small $^{a}$ School of Mathematics, East China University of Science and Technology, \\
\small  Shanghai 200237, China\\
\small $^{b}$ School of Mathematics and Finance, Chuzhou University, \\
\small  Chuzhou, Anhui 239012, China\\
}

\date{}
\maketitle
{\flushleft\large\bf Abstract}
Let ${\rm ex}(n,F)$ and ${\rm spex}(n,F)$ be the maximum size
and maximum spectral radius of an $F$-free graph of order $n$, respectively.
The value ${\rm spex}(n,F)$ is called the spectral extremal value of $F$.
Nikiforov [J. Graph Theory 62 (2009) 362--368] gave the spectral Stability Lemma,
which implies that for every $\varepsilon>0$, sufficiently large $n$ and a non-bipartite graph $H$ with chromatic number $\chi(H)$, the extremal graph for ${\rm spex}(n,H)$ can be obtained from the Tur\'{a}n graph $T_{\chi(H)-1}(n)$ by adding and deleting at most $\varepsilon n^2$ edges.
It is still a challenging problem to determine the exact spectral extremal values of many non-bipartite graphs.
Given a graph $F$ and an integer $p\geq 2$, the edge blow-up of $F$, denoted by $F^{p+1}$,
is the graph obtained from replacing each edge in $F$ by a $K_{p+1}$ where the new vertices of $K_{p+1}$ are all distinct.
In this paper, we determine the exact spectral extremal values of the edge blow-up of all non-bipartite graphs and provide the asymptotic spectral extremal values of the edge blow-up of all bipartite graphs for sufficiently large $n$,
which can be seen as a spectral version of the theorem on ${\rm ex}(n,F^{p+1})$ given by Yuan [J. Combin. Theory Ser. B 152 (2022) 379--398].
As applications, on the one hand, we generalize several previous results on ${\rm spex}(n,F^{p+1})$ for $F$ being a matching and a star for $p\geq 3$.
On the other hand, we obtain the exact values of ${\rm spex}(n,F^{p+1})$ for $F$ being
a path, a cycle and a complete graph.

\begin{flushleft}
\end{flushleft}

\section{Introduction}

\subsection{Prerequisite}
Given a graph family  $\mathcal{F}$, we call a graph \emph{$\mathcal{F}$-free}
if it does not contain any copy of $F\in \mathcal{F}$ as a subgraph.
The classic \emph{Tur\'{a}n's problem} aims to determine the maximum size in any $\mathcal{F}$-free graph of order $n$.
The aforementioned value is called the \emph{Tur\'{a}n number} of $\mathcal{F}$ and denoted by ${\rm ex}(n,\mathcal{F})$.
An $\mathcal{F}$-free graph is said to be \textit{extremal} for ${\rm ex}(n,\mathcal{F})$, if it has $n$ vertices and ${\rm ex}(n,\mathcal{F})$ edges.
Denote by ${\rm EX}(n,\mathcal{F})$ the family of extremal graphs for  ${\rm ex}(n,\mathcal{F})$.
For convenience, we use ${\rm ex}(n,F)$ (resp. ${\rm EX}(n,F)$) instead of ${\rm ex}(n,\mathcal{F})$ (resp. ${\rm EX}(n,\mathcal{F})$)
if $\mathcal{F}=\{F\}$.
As one of the earliest results in extremal graph theory, Tur\'{a}n's theorem \cite{Turan} states that
$T_p(n)$ is the extremal graph for ${\rm ex}(n,K_{p+1})$, where $T_p(n)$ denotes the complete $p$-partite graph on $n$ vertices with part sizes as equal as possible.
Many results are known in this area (see, for example, \cite{E1962,EFGG1995,EG1959,FS1975,FG2015,FS2013}).

Given a graph $G$, let $A(G)$ be its  adjacency matrix, and $\rho(G)$ be its spectral radius.
The spectral extremal value of a graph family $\mathcal{F}$, denoted by ${\rm spex}(n,\mathcal{F})$,
is the maximum spectral radius in any $\mathcal{F}$-free graph of order $n$.
An $\mathcal{F}$-free graph is said to be \textit{extremal} for ${\rm spex}(n,\mathcal{F})$, if it has $n$ vertices and its spectral radius is equal to ${\rm spex}(n,\mathcal{F})$.
Denote by ${\rm SPEX}(n,\mathcal{F})$ the family of extremal graphs for  ${\rm spex}(n,\mathcal{F})$.
For convenience, we use ${\rm spex}(n,F)$ (resp. ${\rm SPEX}(n,F)$) instead of ${\rm spex}(n,\mathcal{F})$ (resp. ${\rm SPEX}(n,\mathcal{F})$)
if $\mathcal{F}=\{F\}$.
In recent years, the investigation on ${\rm spex}(n,F)$ has become very popular (see, for example, \cite{Cioaba2,Cioaba0,Li,LIY,LIN1,LIN2,Nikiforov1,TAIT2,TAIT1,ZHAI}).
In 2022, Cioab\u{a}, Desai and Tait gave the following conjecture.

\begin{conj}\emph{(\cite{Cioaba2})}\label{conj2.1}
Let $p\geq 2$ and $F$ be any graph such that every graph in ${\rm EX}(n,F)$ is obtained from $T_p(n)$ by adding $O(1)$ edges.
Then, ${\rm SPEX}(n,F)\subseteq {\rm EX}(n,F)$ for sufficiently large $n$.
\end{conj}

The validity of Conjecture \ref{conj2.1} for $K_{p+1}$ was proved by Nikiforov \cite{Nikiforov5},
for friendship graph was proved by Cioab\u{a}, Feng, Tait and Zhang \cite{CFTZ2020}, for intersecting cliques was proved by Desai,  Kang, Li, Ni, Tait and Wang \cite{DKL2022}, and for intersecting odd cycles was proved by Li and Peng \cite{Li}.
Recently, Wang, Kang and Xue \cite{WANG} completely solved Conjecture \ref{conj2.1}.
Naturally, one may consider a variation by the following question.

\begin{ques}\label{prob1.1}
Which graph $F$ satisfying that every graph in ${\rm EX}(n,F)$ is obtained from Tur\'{a}n graph by adding $O(n)$ edges, can guarantee ${\rm SPEX}(n,F)\subseteq {\rm EX}(n,F)$ for sufficiently large $n$?
\end{ques}

Given a graph $F$ and an integer $p\geq 2$, the edge blow-up of $F$, denoted by $F^{p+1}$,
is the graph obtained from replacing each edge in $F$ by a $K_{p+1}$ where the new vertices of $K_{p+1}$ are all distinct.
Denote by $S_t$ a star on $t$ vertices and $M_{2t}$ the vertex-disjoint union of $t$ independent edges.
The validity of Question \ref{prob1.1}
for $S_{t+1}^{p+1}$ was proved by Cioab\u{a}, Feng, Tait and Zhang \cite{CFTZ2020} when $p=2$ and Desai, Kang, Li, Ni, Tait and Wang \cite{DKL2022} when $p\geq 3$,  for $M_{2t}^{p+1}$ was proved by Ni, Wang and Kang \cite{NWK2023}.
This motivates us to consider the validity of Question \ref{prob1.1} for $F^{p+1}$ with more general graphs $F$.
Yuan \cite{Y2022} showed that if $F$ is bipartite with $q<q(F)$ or $F$ is non-bipartite,
then for sufficient large $n$, every graph in ${\rm EX}(n,F^{P+1})$ is obtained from Tur\'{a}n graph by adding $O(n)$ edges (where parameters $\mu$, $\chi(F)$, $q$ and $q(F)$ with respect to $F$ alone are defined in Subsection \ref{sub1.2}, and more details can be seen in Theorems \ref{theorem3.1} and \ref{theorem3.2}).
Inspired by the work of Yuan, we give a partial answer to Question \ref{prob1.1} for $F^{p+1}$ as follows.
\begin{thm}\label{thm1.001}
Let $F$ be a graph and $p\geq \max\{4\mu-2, \chi(F)+1\}$.
If $F$ is bipartite with $q<q(F)$ or $F$ is non-bipartite,
then ${\rm SPEX}(n,F^{p+1})\subseteq {\rm EX}(n,F^{p+1})$ for sufficiently large $n$.
\end{thm}

\subsection{Definitions and Notations} \label{sub1.2}

To state our main theorems and related results, we first introduce some definitions and notations.
Given two disjoint vertex subsets $X,Y\subseteq V(G)$.
Let $G[X]$ be the subgraph induced by $X$, $G-X$
be the subgraph induced by $V(G)\setminus X$,
and $G[X,Y]$ be the bipartite subgraph on the vertex set $X\cup Y$
which consists of all edges with one
endpoint in $X$ and the other in $Y$.
For short, we write $e(X)=e(G[X])$ and $e(X,Y)=e(G[X,Y])$, respectively.
Let $F+H$ be the join and $F\cup H$ be the union, of $F$ and $H$, respectively.
Define $E_t$ as the empty graph on $t$ vertices.
Given a graph family  $\mathcal{H}$, set $p(\mathcal{H})=\min\{\chi(H)~|~H\in \mathcal{H}\}-1$,
where $\chi(H)$ denotes the chromatic number of $H$.
To study ${\rm SPEX}(n,F^{p+1})$, we first introduce the decomposition family $\mathcal{M}(\mathcal{H})$,
which was given by Simonovits \cite{Simonovits1974}.

\begin{definition}\label{def2.1}
Given a graph family $\mathcal{H}$ with $p(\mathcal{H})=p\geq 2$,
 let $\mathcal{M}(\mathcal{H})$ be the family of minimal graphs $M$ satisfying that
 there exist an $H\in \mathcal{H}$ and a constant $t=t(\mathcal{H})$ such that
    $H\subseteq (M\cup E_t)+T_{p-1}((p-1)t)$.
We call $\mathcal{M}(\mathcal{H})$ the decomposition family of $\mathcal{H}$.
\end{definition}

This definition provides us a useful approach to characterize the extremal graph for ${\rm spex}(n,\mathcal{H})$ via the Tur\'{a}n graph $T_p(n)$,
that is, one can embed a maximal $\mathcal{M}(\mathcal{H})$-free graph into one class of $T_p(n)$ to obtain the extremal graph for ${\rm spex}(n,\mathcal{H})$.
Note that there exists an $H\in \mathcal{H}$ with $\chi(H)=p+1$.
Then $H\subseteq T_{p+1}((p+1)s)$ for some positive integer $s$.
This implies that $\mathcal{M}(\mathcal{H})$ always contains at least one bipartite graph.

Given a graph $F$,
for brevity, we write
$\mathcal{M}$ instead of $\mathcal{M}(\{F^{p+1}\})$.
A vertex split on some vertex $v\in V(F)$ is defined as follows:
replace $v$ by an independent set of size $d_F(v)$ in which each vertex is adjacent to exactly one distinct vertex in $N_F(v)$.
Let $F\nabla U$ be the graph obtained from $F$ by applying a vertex split on some vertex subset $U\subseteq V(F)$,
and let $\mathcal{F}(F)=\{F\nabla U~|~U\subseteq V(F)\}$.
It is easy to see that $F\nabla V(F)=M_{2e(F)}$ and $F\nabla\varnothing=F$.
The following lemma concerning the decomposition family $\mathcal{M}$ is due to Liu \cite{L2013}.

\begin{lem}\label{lemma3.01}\emph{(\cite{L2013})}
Given a graph $F$ with $2\leq \chi(F) \leq p-1$, we have $\mathcal{M}=\mathcal{F}(F)$.
In particular, $M_{2e(F)}\in\mathcal{M}$.
Moreover, if $M\in \mathcal{M}$, then after splitting any vertex set of $M$, the resulting graph also belongs to $\mathcal{M}$.
\end{lem}

We need the following parameters $\beta, q, \mathcal{B}, \mu$ and $\lambda$ concerning $\mathcal{M}$ to state our main result.
A \emph{covering} of a graph is a set of vertices which meets all edges of the graph.
Set
            $$\beta=\min\{\beta(M)~|~M\in \mathcal{M}\},$$
where $\beta(M)$ is the minimum number of vertices in a covering of $M$.
An \emph{independent set} of a graph is a set of vertices any two of which are non-adjacent.
Similarly, an \emph{independent covering} of a bipartite graph is an independent set which meets all edges.
Set
            $$q=\min\{q(M)~|~M\in \mathcal{M}~~\text{is bipartite} \},$$
where $q(M)$ denotes the minimum number of vertices in an independent covering of $M$.
Clearly, $\beta\leq q$.
If $\beta=q$,
then set $\mathcal{B}=\{K_q\}$ and otherwise,
   $$\mathcal{B}=\{M[S]~|~M\in \mathcal{M},~S~\text{is a covering of}~M~\text{with}~|S|<q\}.$$

By the definition of $q$ and Lemma \ref{lemma3.01}, there exists an $U^*\subseteq V(F)$ such that
$F\nabla U^*$ is bipartite and $L_{U^*}$ is an independent covering of $F\nabla U^*$ with $|L_{U^*}|=q$.
Set $I_{U^*}:=L_{U^*}\cap (V(F)\setminus U^*)$. More precisely, $I_{U^*}$ is the set of vertices in $L_{U^*}$ which are not obtained by splitting vertices in $U^*$.
Thus, $I_{U^*}$ is an independent set of $F$.
Let $F'$ be the graph obtained from $F\nabla U^*$ by splitting all vertices in $V(F\nabla U^*)\setminus L_{U^*}$.
Then $F'$ is a star forest and $q(F')=q$.
Set $U'= V(F)\setminus I_{U^*}$.
We can further obtain that $F'\cong F\nabla U'\cong (\bigcup_{v\in I_{U^*}}S_{d_F(v)+1})\cup M_{2e(U')}$ and $q=|I_{U^*}|+e(U')$.
Let $\mathcal{U}$ be the set of vertex subsets $U\subseteq V(F)$ such that $V(F)\setminus U$ is an independent set of $F$ and $q(F\nabla U)=q$.
Clearly, $U'\in \mathcal{U}$, and hence $\mathcal{U}$ is non-empty.
For any $U\in \mathcal{U}$, since $e(F)>0$ and $V(F)\setminus U$ is an independent set of $F$,
we can see that $U$ is non-empty.
Set $$\mu=\min_{U\in \mathcal{U}}\max\{d_{F}(x)~|~x\in U\}.$$

Define $\mathcal{M}^*$ as the family of bipartite graphs $M\in \mathcal{M}$ with $q(M)=q$.
Let $$\lambda=\min_{M\in \mathcal{M}^*}\min\{d_{M}(x)~|~x\in L,~
L~\text{is an independent covering of}~M~\text{with}~|L|=q\}.$$

Given a simple graph $G$,
denote by $e(G)$ the number of edges, $\nu(G)$ the matching number, and
$\Delta(G)$ the maximum degree, respectively.
Define $f(\nu,\Delta)=\max\{e(G)~|~\nu(G)\leq \nu, \Delta(G)\leq \Delta\}$.
Chv\'{a}tal and Hanson \cite{CH1976} proved the following theorem.

\begin{thm}\emph{(\cite{CH1976})}\label{theorem2.1}
For every $\nu\geq 1$ and $\Delta\geq 1$,
   $$f(\nu,\Delta)=\nu\Delta+\left\lfloor\frac{\Delta}{2}\right\rfloor
           +\left\lfloor\frac{\nu}{\lceil\Delta/2\rceil}\right\rfloor\leq \nu \Delta +\nu.$$
\end{thm}

Let $\mathcal{D}_{\lambda-1}$ denote the family of extremal graphs for $f(\lambda-1,\lambda-1)$ in  Theorem \ref{theorem2.1}.
In 1972, Abbott, Hanson and Sauer \cite{ AHS1972} determined $f(\lambda-1,\lambda-1)$.
Moreover, they showed that $\mathcal{D}_{\lambda-1}=\{2K_{\lambda}\}$ when $\lambda$ is odd,
and $\mathcal{D}_{\lambda-1}$ is the family of all the graphs with $2\lambda-1$ vertices,
$\lambda^2-\frac32\lambda$ edges and maximum degree $\lambda-1$ when $\lambda$ is even.


Let $H(n,p,q)=E_{q-1}+T_p(n-q+1)$.
Denote by $\mathcal{H}(n,p,q,\lambda-1,\mathcal{B})$ the family of graphs which are obtained from $H(n,p,q)$ by  embedding a copy of $D_{\lambda-1}\in \mathcal{D}_{\lambda-1}$ into one class of $T_p(n-q+1)$,
and embedding a copy of $Q_{q-1}\in {\rm EX}(q-1,\mathcal{B})$ into $E_{q-1}$.
For convenience, we use $\mathcal{H}(n,p,q,\lambda-1,B)$ to denote $\mathcal{H}(n,p,q,\lambda-1,\mathcal{B})$ if $\mathcal{B}=\{B\}$.

\subsection{Main result and its applications}

The main result on ${\rm spex}(n,F^{p+1})$ is as follows.

\begin{thm}\label{theorem2.2}
Let $F$ be a graph and $p\geq \max\{4\mu-2, \chi(F)+1\}$, and let $H(n,p,q,\lambda-1,\mathcal{B})$ be any graph in $\mathcal{H}(n,p,q,\lambda-1,\mathcal{B})$. If $n$ is sufficiently large, then we have the following:\\
(i) If ${\rm ex}(n,F^{p+1})<e(H(n,p,q,\lambda-1,\mathcal{B}))$, then
$$\rho(H(n,p,q,0,\mathcal{B}))\leq {\rm spex}(n,F^{p+1})< \rho(H(n,p,q,\lambda-1,\mathcal{B})).$$
(ii) If ${\rm ex}(n,F^{p+1})=e(H(n,p,q,\lambda-1,\mathcal{B}))$, then
 $${\rm SPEX}(n,F^{p+1})\subseteq {\rm EX}(n,F^{p+1})\subseteq \mathcal{H}(n,p,q,\lambda-1,\mathcal{B}).$$
\end{thm}


Define a path, a cycle and a complete graph on $t$ vertices by $P_t,C_t$ and $K_{t}$, respectively.
Now we provide some applications of Theorem \ref{theorem2.2} to these special graphs.

{\bf Example 1.} $F=M_{2t}$.
Clearly, $\mathcal{M}=\{M_{2t}\}$.
Then $q=\beta=t$, $\lambda=1$, $\mu=1$, $\chi(F)=2$ and $\mathcal{B}=\{K_t\}$.
The extremal graph for ${\rm ex}(n,M_{2t}^{p+1})$ was determined
by Erd\H{o}s \cite{E1962} for $p=2$, and by Moon \cite{M1968} and Simonovits \cite{S1968} independently for general $p\geq 2$: for two integers $p,t\geq 2$ and sufficiently large integer $n$, ${\rm EX}(n,M_{2t}^{p+1})=\mathcal{H}(n,p,t,0,K_t)$.
Combining these with Theorem \ref{theorem2.2}, we have the following result.
\begin{cor}\label{cor4.1}\emph{(\cite{NWK2023})}
For two integers $t\geq 2,p\geq 3$ and sufficiently large integer $n$,
 $${\rm SPEX}(n,M_{2t}^{p+1})={\rm EX}(n,M_{2t}^{p+1})=\mathcal{H}(n,p,q,0,K_t)=\{K_{t-1}+T_p(n-t+1)\}.$$
\end{cor}

{\bf Example 2.} $F=S_{t+1}$.
Clearly, $\mathcal{M}=\{S_{t+1}, M_{2t}\}$ and $\mathcal{M}^*=\{S_{t+1}\}$.
Then $q=\beta=1$, $\lambda=t$, $\mu=1$, $\chi(F)=2$ and $\mathcal{B}=\{K_1\}$.
The extremal graph for ${\rm ex}(n,S_{k+1}^{p+1})$ was determined by
Erd\H{o}s, F\"{u}redi, Gould and Gunderson \cite{EFGG1995} for $p=2$, and by Chen, Gould, Pfender and Wei \cite{CGPW2003} for general $p\geq 3$: for two integers $p,t\geq 2$ and sufficiently large integer $n$,
 ${\rm EX}(n,S_{t+1}^{p+1})=\mathcal{H}(n,p,1,t-1,K_1)$.
Combining these with Theorem \ref{theorem2.2},
we have the following result.

\begin{cor}\label{cor4.2}\emph{(\cite{DKL2022})}
For two integers $t\geq 2,p\geq 3$ and sufficiently large integer $n$,
 $${\rm SPEX}(n,S_{t+1}^{p+1})\subseteq{\rm  EX}(n,S_{t+1}^{p+1})=\mathcal{H}(n,p,1,t-1,K_1).$$
\end{cor}

{\bf Example 3.} $F=P_t$.
By Lemma \ref{lemma3.01}, $\mathcal{M}$ consists of all linear forests with $t-1$ edges.
Let $F:=\bigcup_{i=1}^{k}P_{\ell_i}$ be an arbitrary linear forests in $\mathcal{M}$.
Then $$q(F)=\sum_{i=1}^{k}\left\lfloor\frac{\ell_i}{2}\right\rfloor
     \geq\sum_{i=1}^{k}\frac{\ell_i-1}{2}=\frac{t-1}{2}.$$
If $t$ is even, then $(\frac{t}{2}-1)S_3\cup S_2\in \mathcal{M}$.
Thus, $q=\beta=t/2$, $\lambda=1$, $\mu=2$, $\chi(F)=2$ and $\mathcal{B}=\{K_{t/2}\}$.
If $t$ is odd, then $\frac{t-1}{2}S_3\in \mathcal{M}$.
Thus, $q=\beta=(t-1)/2$, $\lambda=2$, $\mu=2$, $\chi(F)=2$ and $\mathcal{B}=\{K_{(t-1)/2}\}$.
Glebov \cite{G1111} characterized the extremal graph for ${\rm ex}(n,P_t^{p+1})$,
which was generalized by Liu \cite{L2013} to edge blow-up of paths and a class of trees:
let $t,p,n$ be three integers with $t\geq 2,p\geq 3$ and $n$ sufficiently large.
Then ${\rm EX}(n,P_t^{p+1})=\mathcal{H}(n,p,\lfloor t/2\rfloor,i,K_{\lfloor t/2\rfloor})$, where $i=1$ when $t$ is odd and $i=0$ otherwise.
Combining these with Theorem \ref{theorem2.2}, we have the following result.

\begin{cor}\label{cor4.3}
Let $t,p,n$ be three integers with $t\geq 2,p\geq 6,$ and $n$ sufficiently large. Then
 $${\rm  SPEX}(n,P_t^{p+1})\subseteq{\rm  EX}(n,P_t^{p+1})=\mathcal{H}(n,p,\lfloor t/2\rfloor,i,K_{\lfloor t/2\rfloor}),$$
 where $i=1$ when $t$ is odd and $i=0$ when $t$ is even.
\end{cor}

{\bf Example 4.} $F=C_t$.
By Lemma \ref{lemma3.01}, $\mathcal{M}$ consists of all linear forests with $t$ edges and $C_t$.
Let $F:=\bigcup_{i=1}^{k}P_{\ell_i}$ be an arbitrary linear forests in $\mathcal{M}$.
Then $$q(F)=\sum_{i=1}^{k}\left\lfloor\frac{\ell_i}{2}\right\rfloor
     \geq\sum_{i=1}^{k}\frac{\ell_i-1}{2}=\frac{t}{2}.$$
We can further find that $q(F)=\lceil t/2\rceil$ when $F$ is a cycle.
If $t$ is even, then $\frac{t}{2}S_3\in \mathcal{M}$.
It follows that $q=\beta=t/2$, $\lambda=2$, $\mu=2$, $\chi(F)=2$ and $\mathcal{B}=\{K_{t/2}\}$.
If $t$ is odd, then $\frac{t-1}{2}S_3\cup S_2\in \mathcal{M}$.
It follows that $q=\beta=(t+1)/2$, $\lambda=1$, $\mu=2$, $\chi(F)=3$ and $\mathcal{B}=\{K_{(t+1)/2}\}$.
Liu \cite{L2013} showed that ${\rm  EX}(n,C_t^{p+1})=\mathcal{H}(n,p,t/2,1,K_{t/2})$ when $t$ is even and $p\geq 3$, and
${\rm  EX}(n,C_t^{p+1})=\mathcal{H}(n,p,(t+1)/2,0,K_{(t+1)/2})$ when $t$ is odd and $p\geq 4$.
Combining these with Theorem \ref{theorem2.2},
we have the following result.

\begin{cor}\label{cor4.4}
Let $t,p,n$ be three integers with $t\geq 3, p\geq 6$ and $n$ sufficiently large.
Then
 $${\rm SPEX}(n,C_t^{p+1})\subseteq{\rm  EX}(n,C_t^{p+1})=\mathcal{H}(n,p,\lfloor (t+1)/2\rfloor,i,K_{\lfloor(t+1)/2\rfloor}),$$
where $i=1$ when $t$ is even and $i=0$ when $t$ is odd.
\end{cor}

{\bf Example 5.} $F=K_t$.
Denote by $S_{k,k}$ the graph on $2k$ vertices obtained from two copies of $S_k$
by joining the centers with a new edge.
Since each bipartite graph in $\mathcal{M}$ is obtained by splitting at least $t-2$ vertices of $K_t$,
$\mathcal{M}^*=\{S_{t-1,t-1}\cup \binom{t-2}{2}K_2,S_t\cup \binom{t-1}{2}K_2\}$.
Then $q=\binom{t-2}{2}+t-1=\binom{t-1}{2}+1$, $\beta=2+\binom{t-2}{2}$, $\lambda=1$, $\mu=t-1$, $\chi(F)=t$ and $\mathcal{B}=\{K_{2}\cup \binom{t-2}{2}K_1\}$.
As an application of Theorem \ref{theorem3.2}, Yuan \cite{Y2022} showed that
 $${\rm  EX}(n,K_t^{p+1})=\mathcal{H}\left(n,p,\binom{t-1}{2}+1,0,K_{2}\cup \binom{t-2}{2}K_1\right)=\left\{H\left(n,p,\binom{t-1}{2}+1\right)\right\}$$
for $p\geq t+1\geq 4$ and sufficiently large $n$.
Combining these with Theorem \ref{theorem2.2},
we get the following.

\begin{cor}\label{cor4.5}
Let $t,p,n$ be three integers with  $t\geq 3,p\geq 4t-6,$ and $n$ sufficiently large.
Then
   $${\rm SPEX}(n,K_t^{p+1})={\rm  EX}(n,K_t^{p+1})= \left\{H\left(n,p,\binom{t-1}{2}+1\right)\right\}.$$
\end{cor}

\section{Proof of Theorem \ref{theorem2.2}}\label{section4}

In this section, we first list some theorems and lemmas that will be used in later proof of  Theorem  \ref{theorem2.2}.
The following extremal results on edge blow-up graphs are due to Yuan (\cite{Y2022}).

\begin{thm}\label{theorem3.1}\emph{(\cite{Y2022})}
Let $F$ be a bipartite graph and $p\geq3$.
If $n$ is sufficiently large, then we have the following:\\
(i) If $q=q(F)$, then $e(H(n,p,q,0,\mathcal{B}))\leq {\rm ex}(n,F^{p+1})\leq e(H(n,p,q,\lambda-1,K_q))$;\\
(ii) If $q<q(F)$, then ${\rm EX}(n,F^{p+1})=\mathcal{H}(n,p,q,0,\mathcal{B})$.
\end{thm}

\begin{thm}\label{theorem3.2}\emph{(\cite{Y2022})}
Let $F$ be a non-bipartite graph and $p\geq \chi(F)+1$.
If $n$ is sufficiently large, then ${\rm EX}(n,F^{p+1})=\mathcal{H}(n,p,q,0,\mathcal{B})$.
\end{thm}

Moreover, from the proofs of Theorems \ref{theorem3.1} and \ref{theorem3.2}, we can see that
\begin{equation}\label{align0001}
H(n,p,q)~\text{and}~H(n,p,q,0,\mathcal{B})~\text{are}~F^{p+1}\text{-free~for~any}~n,
\end{equation}
and
\begin{equation}\label{align1001}
{\rm ex}(n,F^{p+1})\leq e(H(n,p,q,\lambda-1,\mathcal{B}))~\text{for sufficiently large}~n~\text{and}~p\geq \chi(F)+1,
\end{equation}
with equality if and only if ${\rm EX}(n,F^{p+1})\subseteq\mathcal{H}(n,p,q,\lambda-1,\mathcal{B})$.

The following is the spectral version of the Stability Lemma due to Nikiforov \cite{Nikiforov4}.

\begin{lem}\label{lemma3.001}\emph{(\cite{Nikiforov4})}
Let $p\geq 2$, $\frac{1}{\ln n}<c<p^{-8(p+21)(p+1)}$, $0<\varepsilon<2^{-36}p^{-24}$, and $G$ be an $n$-vertex graph.
If $\rho(G)>(\frac{p-1}{p}-\varepsilon)n$, then one of the following holds:\\
(i) $G$ contains a complete $(p+1)$-partite graph $K_{\lfloor c\ln n\rfloor, \dots,\lfloor c\ln n\rfloor,\lceil n^{1-\sqrt{c}}\rceil}$;\\
(ii) $G$ differs from $T_p(n)$ in fewer than $(\varepsilon^{\frac{1}{4}}+c^{\frac{1}{8p+8}})n^2$ edges.
\end{lem}

From the above lemma,
Desai et al. \cite{DKL2022} obtained the following result,
which helps us to  present an approach to prove Theorem  \ref{theorem2.2}.

\begin{lem} \label{lemma3.002}\emph{(\cite{DKL2022})}
Let $H$ be a graph with $\chi(H)=p+1$.
For every $\varepsilon>0$, there exist $\delta>0$ and $n_0$ such that
if $G$ is an $H$-free graph on $n\geq n_0$ vertices with $\rho(G)\geq(\frac{p-1}{p}-\delta)n$,
then $G$ can be obtained from $T_p(n)$ by adding and deleting at most $\varepsilon n^2$ edges.
\end{lem}

In what follows, we will give the  proof of Theorem \ref{theorem2.2}.
If $F$ contains some isolated vertices, then let $F_0$ be the graph obtained from $F$ by deleting these isolated vertices.
It is not hard to verify that a graph is $F^{p+1}$-free if and only if it is $F_0^{p+1}$-free for sufficiently large $n$.
Hence, throughout the paper,
we just consider the graph $F$ without isolated vertices.
Moreover, we always assume that $n$ is sufficiently large and $G$ is an extremal graph for ${\rm spex}(n,F^{p+1})$ in the following.
Clearly, $G$ is connected.
Otherwise, there exist two components $H_1$ and $H_2$ of $G$ with $\rho(H_1)=\rho(G)$.
Then, we can add a cut edge between $H_1$ and $H_2$ to obtain a new graph with larger spectral radius, which gives a contradiction.
By Perron-Frobenius theorem, there exists a positive unit eigenvector
$X=(x_1,\ldots,x_n)^{\mathrm{T}}$ corresponding to $\rho(G)$, where $x_{u^*}=\max\{x_i~|~i\in V(G)\}$.
Set $$\phi=\max\{2e(F),|V(F^{p+1})|,p,10\}~~\text{and}~~\eta<\frac{1}{50\phi^7},$$
which will be frequently used in the sequel.

\begin{lem}\label{lemma3.02}
For sufficiently large $n$, we have $\rho(G)\geq  \frac{p-1}{p}n-\frac{p}{4n}.$
\end{lem}

\begin{proof}
Obviously, $\chi(F^{p+1})=p+1$, which implies that $T_{p}(n)$ is $F^{p+1}$-free.
By the Rayleigh quotient and $e(T_p(n))\geq \frac{p-1}{2p}n^2-\frac{p}{8}$, we obtain that
\begin{center}
  $\rho(G)\geq \rho(T_{p}(n))\geq \frac{\mathbf{1}^{\mathrm{T}}A(T_{p}(n))\mathbf{1}}{\mathbf{1}^{\mathrm{T}}\mathbf{1}}
  =\frac{2e(T_{p}(n))}{n}\geq \frac{p-1}{p}n-\frac{p}{4n},$
\end{center}
as desired.
\end{proof}

\begin{lem}\label{lemma3.03}
For sufficiently large $n$, we have $e(G)\geq \big(\frac{p-1}{2p}-\eta^3\big)n^2.$
Furthermore, $G$ admits a partition $V(G)=\bigcup_{i=1}^{p}V_i$ such that $\sum_{1\leq i<j\leq p}e(V_i,V_j)$ attains the maximum,
$\sum_{i=1}^{p}e(V_i)\le \eta^3 n^2$ and $\big||V_i|-\frac{n}{p}\big|\leq\eta n$ for each $i\in [p].$
\end{lem}

\begin{proof}
Recall that $\chi(F^{p+1})=p+1$.
By Lemma \ref{lemma3.02}, $\rho(G)\geq \frac{p-1}{p}n-\frac{p}{4n}$.
Then by Lemma \ref{lemma3.002}, $G$ can be obtained from $T_p(n)$ by adding and deleting at most $\varepsilon n^2$ edges.
Let $\varepsilon>0$ be a constant with $\varepsilon<\eta^3$.
We can further obtain that
$e(G)\geq \frac{p-1}{2p}n^2-\eta^3 n^2$,
and there exists a partition $V(G)=\bigcup_{i=1}^{p} U_i$ such that
$\sum_{i=1}^{p}e(U_i)\leq \eta^3 n^2$ and $\big\lfloor\frac{n}{p}\big\rfloor\leq |U_i|\leq \big\lceil\frac{n}{p}\big\rceil$ for each $i\in [p]$.
Furthermore, we could select a new partition $V(G)=\bigcup_{i=1}^{p}V_i$
such that $\sum_{1\leq i<j\leq p}e(V_i,V_j)$ attains the maximum.
This means that $\sum_{i=1}^{p}e(V_i)$ attains the minimum, and hence
\begin{center}
$\sum\limits_{i=1}^{p}e(V_i)\leq \sum\limits_{i=1}^{p}e(U_i)\leq \eta^3 n^2.$
\end{center}

Set $s=\max\{\big||V_i|-\frac{n}{p}\big|~|~i\in [p]\}$. We may assume without loss of generality that $s=||V_1|-\frac{n}{p}|$.
From Cauchy-Schwarz inequality, we know that $(p-1)\sum_{i=2}^{p}|V_i|^2\geq (\sum_{i=2}^{p}|V_i|)^2$.
Then
   $$2\sum_{2\leq i<j\leq p}|V_i||V_j|=\left(\sum_{i=2}^{p}|V_i|\right)^2-\sum_{i=2}^{p}|V_i|^2
   \leq \frac{p-2}{p-1}(n-|V_1|)^2.$$
Consequently,
\begin{eqnarray*}
e(G)
&\leq& \sum_{1\leq i<j\leq p}|V_i||V_j|+\sum_{i=1}^{p}e(V_i)\\
&\leq& |V_1|(n-|V_1|)+\sum_{2\leq i<j\leq p}|V_i||V_j|+\eta^3 n^2\\
&\leq& |V_1|(n-|V_1|)+\frac{p-2}{2(p-1)}(n-|V_1|)^2+\eta^3 n^2\\
&=& -\frac{p}{2(p-1)}s^2+\frac{p-1}{2p}n^2+\eta^3 n^2,
\end{eqnarray*}
where the last equality holds as $s=||V_1|-\frac{n}{p}|$.
On the other hand,
   $$e(G)\geq e(T_{p}(n))-\eta^3 n^2\geq \frac{p-1}{2p}n^2-\frac{p}{8}-\eta^3 n^2
  >\frac{p-1}{2p}n^2-2\eta^3 n^2.$$
Combining the above two inequalities gives $\frac{p}{2(p-1)}s^2<3\eta^3 n^2$.
Thus, $s<\sqrt{\frac{6(p-1)\eta^3}{p}n^2}<\eta n$.
\end{proof}

Given a vertex $v\in V(G)$ and  a vertex subset $X\subseteq V(G)$ (possibly $v\notin X$),
denote by $N_G(v)$ its neighborhood, $d_G(v)$ its degree in $G$, $N_X(v)$ the set of neighbors of $v$ in $X$, respectively. Set $d_X(v)=|N_X(v)|$.
We shall define two vertex subsets $S$ and $W$ of $G$ and provide their characterization.

\begin{lem}\label{lemma3.04}
Let $S=\{v\in V(G)~|~d_G(v)\leq \big(\frac{p-1}{p}-5\eta\big)n\}.$
Then $|S|\leq \eta n$.
\end{lem}

\begin{proof}
Suppose to the contrary that $|S|>\eta n$,
then there exists a subset $S'\subseteq S$ with $|S'|=\lfloor\eta n\rfloor$.
Set $n'=|V(G-S')|=n-\lfloor\eta n\rfloor$. Then $n'+1< (1-\eta)n+2$.
Combining these with $e(G)\geq (\frac{p-1}{2p}-\eta^3\big)n^2$ (see Lemma \ref{lemma3.03}),  we deduce that
\begin{align*}
 e(G-S')&\geq  e(G)-\sum_{v\in S'}d_G(v)\nonumber\\
  &\geq  \big(\frac{p-1}{2p}-\eta^3\big)n^2-\eta n\Big(\frac{p-1}{p}-5\eta\Big)n\nonumber\\
  &=    \big(\frac{p-1}{2p}-\eta^3-\frac{p-1}{p}\eta+5\eta^2\Big)n^2 \nonumber\\
  &>  e(H(n',p,q,\lambda-1,\mathcal{B})),
\end{align*}
where the last inequality holds as $e(H(n',p,q,\lambda-1,\mathcal{B}))-\frac{p-1}{2p}\big(n'+1\big)^2=O(n)$.
By Theorems \ref{theorem3.1} and \ref{theorem3.2}, we can see that $G-S'$ contains a copy of $F^{p+1}$ as a subgraph, a contradiction.
This completes the proof.
\end{proof}

\begin{lem}\label{lemma3.05}
Let $W=\bigcup_{i=1}^{p}W_i$, where $W_i=\{v\in V_i~|~d_{V_i}(v)\geq 2\eta n\}$.
Then $|W|\leq \frac{1}{16}\eta n$.
\end{lem}

\begin{proof}
For every $i\in [p]$,
\begin{center}
  $2e(V_i)=\sum\limits_{v\in V_i}d_{V_i}(v)\geq
\sum\limits_{v\in W_i}d_{V_i}(v)\geq |W_i|\cdot 2\eta n.$
\end{center}
Combining this with Lemma \ref{lemma3.03} gives
\begin{center}
  $\eta^3 n^2\geq \sum\limits_{i=1}^{p}e(V_i)\geq \sum\limits_{i=1}^{p}|W_i|\eta n=|W|\eta n.$
\end{center}
This yields that $|W|\leq \frac{1}{16}\eta n$.
\end{proof}

Set $\overline{V}_i=V_i\setminus (W\cup S)$ for every $i\in [p]$.
Then we give the following result.

\begin{lem}\label{lemma3.06}
Let $i\in [p]$ and $j\in [\phi^2]$. Then \\
(i) for any $u\in \bigcup_{k\in [p]\setminus \{i\}}(W_{k}\setminus S)$, $d_{V_{i}}(u)\geq\big(\frac{1}{p^2}-2\phi\eta\big)n$; \\
(ii) for any $u\in \bigcup_{k\in [p]\setminus \{i\}}\overline{V}_{k}$, $d_{V_{i}}(u)\geq\big(\frac{1}{p}-2\phi\eta\big)n$;\\
(iii) if $u\in \bigcup_{k\in [p]\setminus \{i\}}(W_{k}\setminus S)$ and
    $\{u_1,\dots,u_j\}\subseteq \bigcup_{k\in [p]\setminus \{i\}}\overline{V}_k$,
then there exist at least $\phi$ vertices in $\overline{V}_i$ adjacent to $u_1,\dots,u_j$ (and $u$).
\end{lem}

\begin{proof}
(i) Assume that $u\in W_{k_0}\setminus S$ for some ${k_0}\in [p]\setminus \{i\}$.
Since $V(G)=\bigcup_{i=1}^{p}V_i$ is a partition such that $\sum_{1\leq i<j\leq p}e(V_i,V_j)$ attains the maximum,
$d_{V_{{k_0}}}(u)\leq \frac{1}{p}d_{G}(u)$.
By Lemma \ref{lemma3.03}, $|V_{k}|\leq \big(\frac{1}{p}+\eta\big)n$ for any $k\in [p]$.
These, together with $d_G(u)> \big(\frac{p-1}{p}-5\eta\big)n$ (as $u\notin S$), give that
\begin{center}
  $d_{V_i}(u)
    = d_G(u)-d_{V_{k_0}}(u)-\sum\limits_{k\in [p]\setminus\{i,k_0\}}d_{V_{k}}(u)
    \geq \frac{p-1}{p}d_G(u)- (p-2)\left(\frac{1}{p}+\eta\right)n
    \geq \big(\frac{1}{p^2}-2\phi\eta\big)n. $
\end{center}

(ii)
Assume that $u\in \overline{V}_{k_0}$ for some $k_0\in [p]\setminus \{i\}$.
 Since $u\notin S$, $d_G(u)> \big(\frac{p-1}{p}-5\eta\big)n$.
Moreover, since $u\notin W_{k_0}$, we have $d_{V_{k_0}}(u)<2\eta n$.
Then it follows from $d_{V_k}(u)\leq |V_{k}|\leq \big(\frac{1}{p}+\eta\big)n$ for any $k\in [p]\setminus\{i,k_0\}$ that
\begin{align*}
  d_{V_i}(u)
    &= d_G(u)-d_{V_{k_0}}(u)-\sum\limits_{k\in [p]\setminus\{i,k_0\}}d_{V_{k}}(u)\\
    &\geq \left(\frac{p-1}{p}-5\eta\right)n-2\eta n- (p-2)\left(\frac{1}{p}+\eta\right)n\\
    &\geq \left(\frac{1}{p}-2\phi\eta\right)n.
\end{align*}

(iii) By (i), $d_{V_i}(u)\geq \big(\frac{1}{p^2}-2\phi\eta\big)n$;
and by (ii), $d_{V_i}(u_k)\geq\big(\frac{1}{p}-2\phi\eta\big)n$ for any $k\in \{1,\dots,j\}$.
Thus,
\begin{align*}
\left|N_{V_i}(u)\cap\left(\bigcap_{k=1}^{j}N_{V_i}(u_k)\right)\right|
           &\geq  |N_{V_i}(u)|+\sum_{k=1}^{j}|N_{V_i}(u_k)|-j|V_i|\\
            &>  \left(\frac{1}{p^2}-2\phi\eta\right)n+j\left(\frac{1}{p}-2\phi\eta\right)n
                       -j\left(\frac{1}{p}+\eta\right)n\\
            &=  \left(\frac{1}{p^2}-3\phi^3\eta\right)n\\
            &\geq  |W\cup S|+\phi,
\end{align*}
as $\eta<\frac{1}{50\phi^7}$, $|W\cup S|\leq 2\eta n$, and $n$ is sufficiently large.
Then, there exist at least $\phi$ vertices in $\overline{V}_i$ adjacent to $u_1,\dots,u_j$ (and $u$).
\end{proof}

In the following two lemmas, we focus on proving $S=\varnothing$.

\begin{lem}\label{lemma3.07}
 For every $i\in [p]$, we have $\overline{\nu}_i\leq e(F)-1$,
where  $\overline{\nu}_i=\nu\big(G[\overline{V}_i]\big)$.
Moreover, $G[\overline{V}_i]$ contains
an independent set $\overline{I}_i$ with $|\overline{V}_i\setminus \overline{I}_i|\leq 2e(F)-2$.
\end{lem}

\begin{proof}
By symmetry, we only need to prove $\overline{\nu}_1\leq e(F)-1$.
Suppose to the contrary that $\overline{\nu}_1\geq e(F)$.
Let $u_{1,1}u_{1,2},u_{1,3}u_{1,4},\dots,u_{1,2e(F)-1}u_{1,2e(F)}$ be independent edges in $G[\overline{V}_1]$.
Recall that $\phi\geq 2e(F)$.
We now select a subset $\widehat{V}_1=\{u_{1,1},u_{1,2},\dots,u_{1,\phi}\}\subseteq \overline{V}_1$.
By Lemma \ref{lemma3.06} (iii), there exists a subset
  $\widehat{V}_2=\{u_{2,1},u_{2,2},\dots,u_{2,\phi}\}\subseteq \overline{V}_2$
such that all vertices in $\widehat{V}_2$ are adjacent to all vertices in $\widehat{V}_1$.
Recursively applying Lemma \ref{lemma3.06} (iii), we can select a sequence of subsets $\widehat{V}_2,\dots,\widehat{V}_p$
such that  for all $k\in \{2,\dots,p\}$,
$\widehat{V}_k=\{u_{k,1},u_{k,2},\dots,u_{k,\phi}\}\subseteq \overline{V}_k$ and all vertices in $\widehat{V}_k$ are adjacent to all vertices in $\bigcup_{i=1}^{k-1}\widehat{V}_i$.
Then $G[\bigcup_{i=1}^{p}\widehat{V_i}]$ contains a subgraph obtained by embedding $M_{2e(F)}$ into one class
of $T_p(p\phi)$.
Applying Definition \ref{def2.1} with $M_{2e(F)}\in \mathcal{M}$, $G[\bigcup_{i=1}^{p}\widehat{V}_i]$ contains a copy of $F^{p+1}$ as a subgraph,
which gives a contradiction.

Therefore, $\overline{\nu}_i\leq e(F)-1$.
If $\overline{\nu}_i=0$, then $\overline{V}_i$ is a desired independent set.
If $\overline{\nu}_i>0$, then let $u_{i,1}u_{i,2},\dots,u_{i,2\overline{\nu}_i-1}u_{i,2\overline{\nu}_i}$
be $\overline{\nu}_i$ independent edges of $G[V_i]$ and $\overline{I}_i=\overline{V}_i\setminus\{u_{i,1},u_{i,2},\dots,u_{i,\overline{2\nu_i}}\}.$
Now, if $G[\overline{I}_i]$ contains an edge,
then $\nu\big(G[\overline{V}_i]\big)\geq \overline{\nu}_i+1$,
a contradiction.
So, $\overline{I}_i$ is an independent set of $G[\overline{V}_i]$.
In both cases, we can see that $|\overline{V}_i\setminus \overline{I}_i|\leq 2e(F)-2$.
\end{proof}

Since $|W|\leq\frac{1}{16}\eta n$,
we may select a vertex $v^*$ such that
$x_{v^*}=\max\{x_v~|~v\in V(G)\setminus W\}$.
Recall that $x_{u^*}=\max\{x_v~|~v\in V(G)\}$. Then
$\rho(G)x_{u^*}\le |W|x_{u^*}+(n-|W|)x_{v^*}.$
Combining this with Lemmas \ref{lemma3.02} and \ref{lemma3.05}, we get
\begin{equation}\label{align0002}
x_{v^*}\geq\frac{\rho(G)-|W|}{n-|W|}x_{u^*}\geq \frac{\rho(G)-|W|}{n}x_{u^*}
     >\left(\frac{p-1}{p}-\eta\right)x_{u^*}>\frac{1}{2}x_{u^*}.
\end{equation}
This implies that
\begin{equation}\label{align0101}
|W|x_{u^*}\leq \frac{1}{16}\eta n \cdot 2x_{v^*}=\frac{1}{8}\eta n x_{v^*}.
\end{equation}

Suppose that $v^*\in V_{i_0}$.
Then by the definition of $W$, we have $|N_{\overline{V}_{i_0}}(v^*)|\leq |N_{V_{i_0}}(v^*)|<2\eta n$.
Combining this with \eqref{align0101} and Lemma \ref{lemma3.04}, we have
\begin{align*}
 \rho(G)x_{v^*}&=  \sum_{v\in N_{W\cup S}(v^*)}x_v+
                    \sum_{v\in N_{\overline{V}_{i_0}}(v^*)}x_v+
                    \sum_{v\in N_{\bigcup_{i\in [p]\setminus \{i_0\}}\overline{V}_i}(v^*)}x_v \nonumber\\
              &<  \big(|W|x_{u^*}+|S|x_{v^*}\big)+2\eta nx_{v^*}+\sum_{i\in [p]\setminus\{i_0\}}\sum_{v\in  \overline{V}_{i}\setminus \overline{I}_{i}}x_v+\sum_{i\in [p]\setminus\{i_0\}}\sum_{v\in \overline{I}_i}x_v\nonumber\\
              &\leq  \big(\frac{1}{8}\eta n x_{v^*}+\eta n x_{v^*}\big)+2\eta nx_{v^*}+\frac{1}{8}\eta nx_{v^*}
                   +\sum_{i\in [p]\setminus\{i_0\}}\sum_{v\in \overline{I}_i}x_v,
\end{align*}
where $\overline{I}_i$ is an independent set of
$G[\overline{V}_i]$
such that $\big|\overline{V}_i\setminus \overline{I}_i\big|\leq 2e(F)-2$ (see Lemma \ref{lemma3.07}).
Then,
\begin{align}\label{align0003}
\sum_{i\in [p]\setminus\{i_0\}}\sum_{v\in \overline{I}_i}x_v
\geq \big(\rho(G)-\frac{13}{4}\eta n\big)x_{v^*}.
\end{align}

\begin{lem}\label{lemma3.08}
$S$ is an empty set.
\end{lem}

\begin{proof}
Suppose to the contrary that there exists a vertex $u_0\in S$.
Let $G'$ be the graph obtained from $G$ by deleting all edges incident to $u_0$
and joining all possible edges between $\bigcup_{i\in [p]\setminus\{i_0\}}\overline{I}_i$ and $u_0$.
We first claim that $G'$ is $F^{p+1}$-free.
Otherwise, $G'$ contains a subgraph $H$ isomorphic to $F^{p+1}$.
From the construction of $G'$,
we can see that $u_0\in V(H)$.
Assume that $N_{H}(u_0)=\{u_1,u_2,\dots,u_a\}$,
then $a\leq \phi-1$ and $u_1,u_2,\dots,u_a\in \bigcup_{i\in [p]\setminus\{i_0\}}\overline{I}_i$ by the definition of $G'$.
Clearly, $\overline{I}_i\subseteq \overline{V}_i$ for each $i\in [p]\setminus \{i_0\}$.
By Lemma \ref{lemma3.06} (iii), we can select a vertex $u\in \overline{V}_{i_0}\setminus V(H)$ adjacent to $u_1,u_2,\dots,u_a$.
This implies that $G[(V(H)\setminus \{u_0\})\cup\{u\}]$ contains a copy of $F^{p+1}$, a contradiction.
Therefore, $G'$ is $F^{p+1}$-free.

In what follows, we shall show that $\rho(G')>\rho(G)$.
By the definition of $S$, $d_G(u_0)\leq (\frac{p-1}{p}-5\eta)n$.
Combining this with \eqref{align0101} and Lemma \ref{lemma3.02}, we have
\begin{align*}
\sum_{v\in N_G(u_0)}x_{v}=\sum_{v\in N_{W}(u_0)}x_v+\sum_{v\in N_{G-W}(u_0)}x_v
\le |W|x_{u^*}+d_G(u_0)x_{v^*}<\big(\rho(G)-4\eta n\big)x_{v^*}.
\end{align*}
Combining this with \eqref{align0003} gives
\begin{center}
  $\rho(G')-\rho(G) \geq X^{\mathrm{T}}\big(A(G')-A(G)\big)X
                  = 2x_{u_0}\left(\sum\limits_{i\in [p]\setminus \{i_0\}}\sum\limits_{v\in \overline{I}_i}x_v-\sum\limits_{v\in N_G(u_0)}x_v\right)>0,$
\end{center}
contradicting that $G$ is an extremal graph for ${\rm spex}(n,F^{p+1})$.
Hence, $S=\varnothing$.
\end{proof}

The following lemma is used to give a lower bound for
entries of vertices in $V(G)\setminus W$.

\begin{lem}\label{lemma3.09}
 For any $v\in V(G)\setminus W$, $(1-7\eta)x_{v^*}\leq x_v\leq x_{v^*}$.
 Moreover, 
 $x_v\geq \big(\frac{p-1}{p}-8\eta\big)x_{u^*}$.
\end{lem}

\begin{proof}
By Lemma \ref{lemma3.02}, $\rho(G)>\frac{p-1}{p}n-\frac{p}{4n}>\frac12 n$.
This, together with (\ref{align0003}), leads to that
\begin{align}\label{align0006}
\sum\limits_{i\in [p]\setminus \{i_0\}}\sum\limits_{v\in \overline{I_i}}x_v
\geq\left(\rho(G)-\frac{13}{4}\eta n\right)x_{v^*}>\left(1-\frac{13}{2}\eta\right)\rho(G)x_{v^*}.
\end{align}

For any $v\in V(G)\setminus W$,
by the definition of $v^*$, we have $x_v\leq x_{v^*}$.
Now we prove that $x_v\geq (1-7\eta)x_{v^*}$.
Suppose to the contrary, then there exists a vertex $u_0$ such that $x_{u_0}<(1-7\eta)x_{v^*}$.
Let $G'$ be the graph obtained from $G$ by deleting edges incident to $u_0$
and joining all edges between $\overline{I}:=(\bigcup_{i\in [p]\setminus \{i_0\}}\overline{I}_i)\setminus \{u_0\}$ and $u_0$.
By a similar discussion as in the proof of Lemma \ref{lemma3.08},
we can find that $G'$ is $F^{p+1}$-free.

On the other hand,
since $x_{u_0}<(1-7\eta)x_{v^*}$, by \eqref{align0006},
\begin{center}
  $\sum\limits_{v\in \overline{I}}x_v-\sum\limits_{v\in N_G(u_0)}x_{v}\geq\sum\limits_{i\in [p]\setminus \{i_0\}}\sum\limits_{v\in \overline{I}_i}x_v -x_{u_0}-\rho(G)x_{u_0}>0.$
\end{center}
This leads to that
\begin{center}
  $\rho(G')-\rho(G) \geq  X^{\mathrm{T}}\big(A(G')-A(G)\big)X
                  = 2x_{u_0}\left(\sum\limits_{v\in \overline{I}}x_v-\sum\limits_{v\in N_G(u_0)}x_v\right)>0,$
\end{center}
contradicting the fact that $G$ is an extremal graph for ${\rm spex}(n,F^{p+1})$.
Hence, $x_v\geq (1-7\eta)x_{v^*}$ for any $v\in V(G)\setminus W$.
Moreover, from \eqref{align0002} we know that $x_{v^*}>\frac{p-1}{p}\big(1-\eta\big)x_{u^*}$.
Then, $$x_v\geq (1-7\eta)\frac{p-1}{p}\big(1-\eta\big)x_{u^*}\geq \left(\frac{p-1}{p}-8\eta\right)x_{u^*}.$$
The proof is complete.
\end{proof}

Given an arbitrary integer $i\in [p]$,
denote by $W_i^*=\{v\in V_i~|~d_{V_i}(v)\geq 3\phi^3\eta n\}$, $W^*=\bigcup_{i=1}^{p}W_i^*$,
$V_i^*=V_i\setminus W_i^*$, $G_i=G[V_i^*]$ and $G_{in}=\bigcup_{i=1}^{p}G_i$, respectively.
Obviously, $W_i^*\subseteq W_i$ and $W^*\subseteq W$.
In the following lemmas, we shall give some characterizations of $W^*$ and $G_{in}$.

\begin{lem}\label{lemma3.10}
We have $|W^*|= q-1$. Moreover, $d_G(u)\geq (1-20\phi^5\eta )n$ and $x_u\geq (1-40\phi^5\eta)x_{u^*}$ for any $u\in W^*$.
\end{lem}

\begin{proof}
We first give three claims.
\begin{claim}\label{claim3.1}
Given two integers $i\in [p]$ and $j\in [\phi^2]$.
If $u\in W^*$ and $\{u_1,u_2,\dots,u_j\}\subseteq \bigcup_{k\in [p]\setminus \{i\}}\overline{V}_k$,
then there exist at least $\phi$ vertices in $\overline{V}_i$ adjacent to $u,u_1,\dots,u_j$.
\end{claim}

\begin{proof}

Assume that $u\in W_{k_0}^*$ for some $k_0\in [p]$.
From the definition of $W_{k_0}^*$ we know that $d_{V_{k_0}}(u)\geq 3\phi^3\eta n$;
and from Lemma \ref{lemma3.06} (i) we have $d_{V_k}(u)\geq \big(\frac{1}{p^2}-2\phi\eta\big)n\geq 3\phi^3\eta n$ for any $k\neq k_0$.
Thus, $d_{V_i}(u)\geq 3\phi^3\eta n$.
Moreover, Lemma \ref{lemma3.06} (ii) gives that $d_{V_i}(u_k)\geq \big(\frac{1}{p}-2\phi\eta\big)n$ for any $k\in [j]$.
Combining these with $|V_i|\leq \big(\frac{1}{p}+\eta\big)n$ due to Lemma \ref{lemma3.03}, we deduce that
\begin{align*}
 \left|N_{V_{i}}(u)\cap \left(\bigcap_{k=1}^{j}N_{V_{i}}(u_k)\right)\right|
           &\geq  |N_{V_i}(u)|+\sum\limits_{k=1}^{j}|N_{V_i}(u_k)|-j|V_i|\\
            &> 3\phi^3\eta n+j\left(\frac{1}{p}-2\phi\eta\right)n-j\left(\frac{1}{p}+\eta\right)n\\
            &\geq  |W|+\phi,
\end{align*}
where the last inequality holds as $j\leq \phi^2$, $|W|\leq \frac{1}{16}\eta n$ and $n$ is sufficiently large.
Then, there exist at least $\phi$ vertices in $\overline{V}_i$ adjacent to $u,u_1,\dots,u_j$.
\end{proof}

\begin{claim}\label{claim3.2}
Given integers $i\in [p]$ and $j\in \{2,\dots,p\}$.
Let $Y=\{y_1,y_2,\dots,y_j\}$ be a clique of $G$ with
$y_1\in W_{i}^*$, $y_2\in N_{\overline{V}_{i}}(y_1)$ and $\{y_2,\dots,y_j\}\subseteq \bigcup_{k=1}^{p}\overline{V}_k$.
Then $G$  has $\phi$ copies of $K_{p+1}$ with  $y_1,y_2,\dots,y_j$ as their common neighbors.
\end{claim}

\begin{proof}
Since $y_1,y_2\in V_{i}$, all vertices in $Y$ belong to at most $j-1$ partite sets.
We may assume without loss of generality that $Y\subseteq \bigcup_{k=1}^{j-1}V_k$.
By Claim \ref{claim3.1}, there exists a subset
  $\widehat{V}_j=\{u_{j,1},u_{j,2},\dots,u_{j,\phi}\}\subseteq \overline{V}_{j}$
such that all vertices in $\widehat{V}_j$ are adjacent to all vertices in $Y$.
If $p\geq j+1$,
then recursively applying Claim \ref{claim3.1}, we can select a sequence of subsets $\widehat{V}_{j+1},\dots,\widehat{V}_p$
such that for all $s\in \{j+1,\dots,p\}$ , $\widehat{V}_s=\{u_{s,1},u_{s,2},\dots,u_{s,\phi}\}\subseteq \overline{V}_s$ and all vertices in $\widehat{V}_s$ are adjacent to all vertices in $\big(\bigcup_{k=j}^{s-1}\widehat{V}_{k}\big)\cup Y$.
Let $Y_k=Y\cup \{u_{j,k},\dots,u_{p,k}\}$ for each $k\in [\phi]$.
Clearly, $G[Y_k]\cong K_{p+1}$ and $Y_a\cap Y_b=Y$ for $a\neq b$.
Thus, $G$ has $\phi$ copies of $K_{p+1}$ with $y_1,y_2,\dots,y_j$ as their common neighbors.
\end{proof}

\begin{claim}\label{claim3.3}
Given non-negative integers $r\in [\mu]$ and $s\leq \mu-1$.
Assume $Y=\{u_1,\dots,u_{r+s},v_1,\dots,v_{r}\}$,
where $u_k\in W_{k^*}^*$ for some $k^*\in [p]$ and each $k\in [r]$,
$v_k\in N_{\overline{V}_{k^*}}(u_k)$ for each $k\in [r]$, and $u_{r+k}\in \overline{V}_{(r+k)^*}$ for each $k\in [s]$ if $s\neq 0$.
Then there exist at least $\phi$ vertices adjacent to all vertices in $Y$.
\end{claim}

\begin{proof}
Since $V(G)=\bigcup_{i=1}^{p}V_i$ is a partition such that $\sum_{1\leq i<j\leq p}e(V_i,V_j)$ attains the maximum,
we have $d_{V_{{j^*}}}(u_j)\leq \frac{1}{p}d_{G}(u_j)$ for each $u_j\in W_{j^*}^*$.
Then,
\begin{align}\label{align0100}
 \left|\left(\bigcap_{j=1}^{r+s}N_G(u_{j})\right)\cap\left(\bigcup_{i=1}^{r+s}V_{i^*}\right)\right|\leq \sum_{i=1}^{r+s}\left|\left(\bigcap_{j=1}^{r+s}N_G(u_{j})\right)\cap V_{i^*}\right|\leq \sum_{i=1}^{r+s}|N_{V_{i^*}}(u_{i})|\leq \frac{1}{p}\sum_{i=1}^{r+s}|N_{G}(u_{i})|.
\end{align}
On the other hand, $\left|\bigcap_{i=1}^{r+s}N_G(u_{i})\right|\geq \sum_{i=1}^{r+s}|N_G(u_{i})|-(r+s-1)(n-1)$.
Combining this with \eqref{align0100}, we obtain
\begin{align*}
\left|\left(\bigcap_{i=1}^{r+s}N_G(u_{i})\right)\Big\backslash\left(\bigcup_{j=1}^{r+s}V_{j^*}\right)\right|
   &\geq \frac{p-1}{p}\sum_{i=1}^{r+s}|N_G(u_{i})|-(r+s-1)(n-1)\\
   &\geq \frac{p-1}{p}(r+s)\left(\frac{p-1}{p}-5\eta\right)n-(r+s-1)(n-1)\\
   &\geq \frac{p^2-(r+s)(2p-1)}{p^2}n-5(r+s)\eta n\\
   &\geq \frac{p^2-(2\mu-1)(2p-1)}{p^2}n-5(2\mu-1)\eta n\\
   &\geq \frac{(2\mu-1)n}{p^2}-5(2\mu-1)\eta n~~~~~~~(\text{as}~p\geq 4\mu-2)\\
   &\geq |W|+\frac{n}{2p^2}~~~~~~~~~~~~~~~~~~(\text{as}~\eta<\frac{1}{50\phi^7}).
\end{align*}
This implies that $$\left|\left(\bigcap_{i=1}^{r+s}N_G(u_{i})\right)\cap\left(\bigcup_{k\in [p]\setminus \{i^*|i\in [r+s]\}}\overline{V}_{k^*}\right)\right|\geq \frac{n}{2p^2}.$$
By the Pigeonhole Principle, there is a vertex subset $\widetilde{V}_{i_0}\subset \overline{V}_{i_0}$ such that
$i_0\in [p]\setminus \{i^*~|~i\in [r+s]\}$, $|\widetilde{V}_{i_0}|\geq \frac{n}{2p^3}$ and all vertices in $\widetilde{V}_{i_0}$ are adjacent to all vertices in $\{u_1,u_2,\dots,u_{r+s}\}$.
By Lemma \ref{lemma3.06} (ii), $d_{V_{i_0}}(v_k)\geq\big(\frac{1}{p}-2\phi\eta\big)n$ for any $k\in [r]$.
Thus,
\begin{align*}
\left|\widetilde{V}_{i_0}\cap\left(\bigcap_{k=1}^{r}N_{V_{i_0}}(v_k)\right)\right|
           &\geq  |\widetilde{V}_{i_0}|+\sum_{k=1}^{r}|N_{V_{i_0}}(v_k)|-r|V_{i_0}|\\
            &>  \frac{n}{2p^3}+r\left(\frac{1}{p}-2\phi\eta\right)n
                       -r\left(\frac{1}{p}+\eta\right)n\\
            &\geq  \left(\frac{1}{2p^3}-3\phi^2\eta\right)n\\
            &\geq  \phi,
\end{align*}
as $\eta<\frac{1}{50\phi^7}$ and $n$ is sufficiently large.
Then, there exist at least $\phi$ vertices in $\widetilde{V}_{i_0}$ adjacent to all vertices in $Y$.
\end{proof}

By the definition of $\mu$,
we choose a vertex subset $U\subseteq V(F)$ such that  $V(F)\setminus U$ is an independent set of $F$,
$q(F\nabla U)=q$ and $\mu=\max\{d_{F}(x)~|~x\in U\}$, and subject to this, $|U|$ is as small as possible.
Recall that $U$ is non-empty as $U\in \mathcal{U}$.
Now we claim that in graph $F$, any vertex in $U$ is adjacent to at least one vertex in $V(F)\setminus U$.
Otherwise, there exists a vertex $v_{i_0}\in U$ such that $N_F(v_{i_0})\subseteq U$.
Set $U'=U\setminus \{v_{i_0}\}$.
Clearly, $V(F)\setminus U'$ is an independent set of $F$.
Consequently, $$q(F\nabla U')=|V(F)\setminus U'|+e(U')=|V(F)\setminus U|+1+e(U)-d_F(v_{i_0})=q-d_F(v_{i_0})+1.$$
If $d_F(v_{i_0})=1$, then $q(F\nabla U')=q$, and hence $U'\in \mathcal{U}$.
However, this contradicts the minimality of $U$ as $|U'|<|U|$.
If $d_F(v_{i_0})\geq 2$, then $q(F\nabla U')=q-d_F(v_{i_0})+1<q$, contradicting the definition of $q$.
Hence, the claim holds, which implies that $U$ is a proper subset of $V(F)$.
Set
$$V(F)\setminus U=\{u_1,\dots,u_{a}\}~~\text{and}~~U=\{v_1,\dots,v_b\}.$$
Clearly, $F\nabla U\cong \big(\bigcup_{i=1}^{a}S_{d_F(u_i)+1}\big)\cup M_{2e(U)}$ and $q=a+e(U)$.
Let $F_b$ be the graph obtained from $F$ by the following:\\
(i) for any $u_iv_j\in E(F)$, replace $u_iv_j$ with a $K_3$ by adding a new vertex $u_{ij}$ such that  $V(K_3)=\{u_i,v_j,u_{ij}\}$;\\
(ii) for any $v_iv_j\in E(F)$, replace $v_iv_j$ with a $K_4$ by adding two new vertices $v_{ij},v_{ji}$ such that $V(K_4)=\{v_i,v_j,v_{ij},v_{ji}\}$;\\
(iii) all the new vertices are distinct.

Let $F_0=F_b-U$.
Then, $F_0\cong F\nabla U$.
Set $L=\{u_1,u_2,\dots,u_a\}\cup \{v_{ij}~|~v_iv_j\in E(F),i<j\}$.
Then, $|L|=a+e(U)=q$ and $L$ is an independent covering of $F_0$.
In the following we shall prove that $|W^*|\leq q-1$.
Otherwise, there exists a subset $W^{**}\subseteq W^*$ with $|W^{**}|=q$.
We may assume that $W^{**}=\{u_1^*,u_2^*,\dots,u_a^*\}\cup \{v_{ij}^*~|~v_iv_j\in E(F),i<j\}$.
Choose an arbitrary vertex $u^*\in W^{**}$.
Assume that $u^*\in V_{i^*}$.
Then, $|N_{\overline{V}_{i^*}}(u^*)|\geq |N_{V_{i^*}}(u^*)|-|W|\geq 2\phi^3 n$.
It follows that for any $v_j\in N_F(u_i)$ we can select a vertex $u_{ij}^*\in N_{\overline{V}_{i^*}}(u_i^*)$ for some $i^*$ satisfying $u_i^*\in V_{i^*}$, and for any $v_j\in N_F(v_i)$ with $i<j$ we can select a vertex $v_{ji}^*\in N_{\overline{V}_{i^*}}(v_{ij}^*)$ for some $i^*$ satisfying $v_{ij}^*\in V_{i^*}$ such that all $u_{ij}^*s$ and $v_{ji}^*s$ are distinct.
Let $F_0^*$ be the subgraph of $G$ induced by
 $$\{u_1^*,\dots,u_a^*\}\cup \{u_{ij}^*~|~u_iv_j\in E(F)\}\cup \{v_{ij}^*~|~v_iv_j\in E(F),i<j\}
\cup \{v_{ji}^*~|~v_iv_j\in E(F),i<j\}.$$
It is not hard to verify that $F_0^*$ contains a copy of $F_0$.

Recall that in graph $F_b$, each vertex in $U$ is adjacent to at least one vertex in $V(F)\setminus U$.
By the definition of $U$, for any $v_k\in U$, we have $|\{v_{i}~|~v_iv_k\in E(U),i<k\}|\leq \mu-1$ and
    $$|\{u_i~|~u_iv_k\in E(F)\}\cup \{v_{ik}~|~v_iv_k\in E(F),i<k\}|\leq d_F(v_k)\leq \mu.$$
Recursively applying Claim \ref{claim3.3}, we can select a sequence of vertices $v_1^*,v_2^*,\dots,v_b^*\in \cup_{k=1}^{p}\overline{V}_k$
such that for each $j\in [b]$, $v_j^*$ is adjacent to all vertices in $\{v_{i}^*~|~v_iv_j\in E(U), i<j\}$ and
$$\{u_i^*~|~u_iv_j\in E(F)\}\cup \{u_{ij}^*~|~u_iv_j\in E(F)\}\cup \{v_{ij}^*~|~v_iv_j\in E(F),i<j\}
\cup \{v_{ji}^*~|~v_iv_j\in E(F),i<j\}.$$
It follows that $G[V(F_0^*)\cup \{v_1^*,v_2^*,\dots,v_b^*\}]$ contains a copy  of $F_b$, say $H$.
Applying Claim \ref{claim3.2} on all $K_3s$ and $K_4s$ of $H$, we can see that $G$ contains a copy of $F^{p+1}$, which also gives a contradiction.
Hence, $|W^*|\leq q-1$.

In what follows, we prove that $|W^*|= q-1$.
If $q=1$, then we are done.
It remains the case $q\geq 2$.
Otherwise, $|W^*|\leq q-2$.
Choose an arbitrary vertex $u_0\in V(G_1)$.
Let $G'$ be the graph obtained from $G$ by deleting all edges of $G_{in}\cup G[W^{*}]$
and adding all edges between $u_0$ and $V_1^*\setminus \{u_0\}$.
Obviously, $G'\subset H(pn,p,q)$.
By \eqref{align0001}, $H(pn,p,q)$ is $F^{p+1}$-free, which implies that $G'$ is also $F^{p+1}$-free.
 By a similar discussion as in the proof of Lemma \ref{lemma3.07},
we can see that $\nu(G_{i})\leq e(F)-1\leq \frac{\phi}{2}-1$; and from the definition of $G_i$ we know that $\Delta(G_i)\leq 3\phi^3\eta n$.
By Theorem \ref{theorem2.1}, $e(G_i)\leq f(e(F)-1,\Delta(G_i))\leq \frac32\phi^4\eta n$.
According to Lemmas \ref{lemma3.09} and \ref{lemma3.03}, we have $x_v>\big(\frac{p-1}{p}-8\eta\big)x_{u^*}\geq \frac{1}{2}x_{u^*}$ for each $v\in V(G)$, and
 $$|V_1^*\setminus \{u_0\}|\geq |V_1|-|W^*|-1\geq |V_1|-q\geq \frac{n}{p}-2\eta n.$$
Combining the above two inequalities, we obtain
\begin{align*}
\rho(G')-\rho(G)
    &\geq  {X^{\mathrm{T}}(A(G')-A(G))X} \nonumber\\
    &\geq  2\left(\sum_{v\in V_1^*\setminus \{u_0\}}x_{u_0}x_{v}-\sum_{vw\in E(G_{in}\cup G[W^*])}x_vx_w\right) \nonumber\\
  &\geq 2\left(\left(\frac{n}{p}-2\eta n\right)\frac{1}{4}x_{u^*}^2-\left(p\cdot\frac32\phi^4\eta n +\binom{q-2}{2}\right)x_{u^*}^2\right)\nonumber\\
  &\geq 2\left(\frac{n}{4p}-2\phi^5\eta n\right) x_{u^*}^2>0
\end{align*}
as $\eta<\frac{1}{50\phi^7}$,
contradicting that $G$ is an extremal graph for ${\rm spex}(n,F^{p+1})$.
Hence $|W^*|= q-1$.

Finally, we consider the case $q\geq 2$, that is, $W^*\neq \varnothing$.
We shall show that $d_G(u)\geq (1-20\phi^5\eta )n$ for any vertex $u\in W^*$.
Otherwise, there exists a vertex $u_0\in W^*$ such that $d_G(u_0)< (1-20\phi^5\eta )n$.
Let $G''$ be the graph obtained from $G$ by deleting all edges of $G_{in}\cup G[W^*]$
and adding edges between $u_0$ and all non-adjacent vertices of $u_0$ in $G-W^*$.
By \eqref{align0001}, $H(pn,p,q)$ is $F^{p+1}$-free,
which implies that $G''$ is also $F^{p+1}$-free as $G''\subset H(pn,p,q)$.
On the other hand,
\begin{align*}
\rho(G'')-\rho(G)
    &\geq  {X^{\mathrm{T}}(A(G'')-A(G))X} \nonumber\\
    &\geq  2\left(\sum_{v\notin N_{G-W^*}(u_0)}x_{u_0}x_{v}-\sum_{vw\in E(G_{in})\cup E(G[W])}x_vx_w\right) \nonumber\\
  &\geq 2\left((20\phi^5\eta n-q+1)\frac{1}{4}x_{u^*}^2-\left(p\cdot\frac32\phi^4\eta n+\binom{q-1}{2}\right)x_{u^*}^2\right)\nonumber\\
  &>0.
\end{align*}
However, this contradicts the fact that $G$ is an extremal graph for ${\rm spex}(n,F^{p+1})$.
Hence, $d_G(u)\geq (1-20\phi^5\eta )n$  for any vertex $u\in W^*$.
Thus,
   $$\rho(G)x_u-\rho(G)x_{u^*}
           \geq \sum\limits_{v\in N_G(u)}x_{v}-\sum\limits_{v\in V(G)}x_{v}
           \geq -\sum\limits_{v\notin N_G(u)}x_{v}
           \geq -20\phi^5\eta n x_{u^*},$$
which implies that $x_u\geq (1-\frac{20\phi^5\eta n}{\rho(G)})x_{u^*}\geq (1-40\phi^5\eta)x_{u^*}$ as $\rho(G)\geq  \frac{p-1}{p}n-\frac{p}{4n}>\frac{n}{2}$.
\end{proof}

\begin{lem}\label{lemma3.11}
Given integers $i\in [p]$ and $j\in [\phi^2]$.
If $\{v_1,v_2,\dots,v_j\}\subseteq \bigcup_{k\in [p]\setminus \{i\}}V_k^*$,
then there exist at least $\phi$ vertices in $V_i^*$ adjacent to all vertices in $W^*\cup \{v_1,\dots,v_j\}$.
\end{lem}

\begin{proof}
Choose an arbitrary  $v\in V_{k_0}^*$ for some $k_0\in [p]\setminus \{i\}$.
Then, $d_{V_{k_0}}(v)<3\phi^3\eta n$.
Since $S=\varnothing$, $d_G(v)> \big(\frac{p-1}{p}-5\eta\big)n$.
Since $d_{V_k}(v)\leq |V_{k}|\leq \big(\frac{1}{p}+\eta\big)n$ for any $k\in [p]\setminus\{i,k_0\}$, we have
\begin{align*}
  d_{V_{i}}(v)
    &= d_G(v)-d_{V_{k_0}}(v)-\sum\limits_{k\in [p]\setminus\{i,k_0\}}d_{V_{k}}(v)\\
    &\geq \left(\frac{p-1}{p}-5\eta\right)n-3\phi^3\eta n- (p-2)\left(\frac{1}{p}+\eta\right)n\\
    &\geq \left(\frac{1}{p}-4\phi^3\eta\right)n.
\end{align*}
Hence,
\begin{align*}
\left|\bigcap_{k=1}^{j}N_{V_{i}}(v_k)\right|
           &\geq  \sum_{k=1}^{j}|N_{V_{i}}(v_k)|-(j-1)|V_{i}|\\
            &>  j\left(\frac{1}{p}-4\phi^3\eta\right)n-(j-1)\left(\frac{1}{p}+\eta\right)n\\
            &>  \left(\frac{1}{p}-5\phi^5\eta\right)n.
\end{align*}
We further obtain that
$$\left|\big(\bigcap_{k=1}^{j}N_{V_{i}}(v_k)\big)\cap \big(\bigcap_{u\in W^*}N_{V_{i}}(u)\big)\right|
   >\left(\frac{1}{p}-5\phi^5\eta\right)n-(q-1)\times20\phi^4\eta n>|W^*|+\phi,$$
where the last inequality holds as $\eta<\frac{1}{50\phi^7}$ and $n$ is sufficiently large.
Then, there exist at least $\phi$ vertices in $V_i^*$ adjacent to all vertices in $W^*\cup \{v_1,\dots,v_j\}$.
\end{proof}

\begin{lem}\label{lemma3.12}
For any $i\in [p]$, $\Delta(G_{i})\leq \lambda-1$ and $e(G_i)\leq \frac12\phi^2$.
\end{lem}

\begin{proof}
By symmetry, we may assume without loss of generality that $i=1$.
By the definition of $\lambda$ and $\mathcal{M}^*$,
there exists an independent covering $L^*$ of some $F^*\in \mathcal{M}^*$ such that $|L^*|=q$ and $\lambda=\min\{d_{F^*}(x)~|~x\in L^*\}$.
More precisely, assume that
$$L^*=\{u_1^*,u_2^*,\dots,u_q^*\}~~\text{and}~~d_{F^*}(u_q^*)=\lambda.$$
Since $L^*$ is an independent covering of $F^*$, $F^*-L^*$ is an empty graph.
Let $F^{**}$ be the graph obtained from $F^*$ by splitting all vertices in $V(F^*)\setminus L^*$.
By Lemma \ref{lemma3.01}, $F^{**}\in \mathcal{M}^*$.
We can further observe that $F^{**}$ is a star forest, in which one of the stars is $S_{\lambda+1}$.

We first show that $\Delta(G_{in})\leq \lambda-1$.
Suppose to the contrary, then we may assume without loss of generality that
there exists a vertex $u_q\in V(G_1)$ with $d_{G_1}(u_q)\geq \lambda$ and
let $u_{1,1},u_{1,2},\dots,u_{1,\lambda}\in N_{G_1}(u_q)$.
By Lemma \ref{lemma3.11}, there exists a subset
  $\widehat{V}_1=\{u_{1,1},u_{1,2},\dots,u_{1,\phi}\}\subseteq V_1^*$
such that all vertices in $\widehat{V}_1\setminus \{u_{1,1},u_{1,2},\dots,u_{1,\lambda}\}$ adjacent to all vertices in $W^*$.
Thus, $G[W^*\cup \widehat{V}_1]$ contains a copy of $F^{**}$.
Recursively applying Lemma \ref{lemma3.11}, we can select a sequence of subsets $\widehat{V}_2,\dots,\widehat{V}_p$
such that  $\widehat{V}_k=\{u_{k,1},u_{k,2},\dots,u_{k,\phi}\}\subseteq V_k^*$, and all vertices in $\widehat{V}_k$ are adjacent to all vertices in $W^*\cup(\bigcup_{i=1}^{k-1}\widehat{V}_i)$ for all $k\in \{2,\dots,p\}$.
By Definition \ref{def2.1}, $G[W^*\cup (\bigcup_{i=1}^{p}\widehat{V}_i)]$ contains a copy of $F^{p+1}$, a contradiction.
Hence,  $\Delta(G_{in})\leq \lambda-1$.
By a similar discussion as in the proof of Lemma \ref{lemma3.07}, $\nu(G_{i})\leq e(F)-1\leq \frac{\phi}{2}-1$.
Combining this with Theorem \ref{theorem2.1}, we have $e(G_i)\leq f(e(F)-1,\Delta(G_i))\leq \frac12\phi^2$.
\end{proof}

\begin{lem}\label{lemma3.13}
If $q\geq 2$, then $e(W^*)\leq e(Q_{q-1})$.
\end{lem}

\begin{proof}
We first give a claim.
\begin{claim}\label{claim3.4}
For any $i\in [p]$, any isolated vertex $u_0$ of $G_i$ is adjacent to all vertices in $V(G)\setminus V(G_i)$.
\end{claim}

\begin{proof}
Otherwise,
let $G'$ be the graph obtained from $G$ by joining an edge between $u_0$ and some vertex $v_1\in (V(G)\setminus V(G_i))\setminus N_G(u_0)$.
Clearly, $\rho(G')>\rho(G)$.

We shall prove that $G'$ is $F^{p+1}$-free.
Otherwise, $G'$ contains a subgraph $H$ isomorphic to $F^{p+1}$.
From the construction of $G'$,
we can see that $u_0\in V(H)$.
Assume that $N_{H}(u_0)=\{u_1,u_2,\dots,u_a\}$,
then $a\leq \phi-1$ and $u_1,u_2,\dots,u_a\in (\bigcup_{k\in [p]\setminus\{i\}}V(G_k))\cup W^*$ by the definition of $G'$.
By Lemma \ref{lemma3.11}, we can select a vertex $u\in V_i^*\setminus V(H)$ adjacent to $u_1,u_2,\dots,u_a$.
Thus $G[(V(H)\setminus \{u_0\})\cup\{u\}]$ contains a copy of $F^{p+1}$, a contradiction.
Therefore, $G'$ is $F^{p+1}$-free.
However, this contradicts the fact that $G$ is an extremal graph for ${\rm spex}(n,F^{p+1})$.
\end{proof}

Suppose to the contrary that $e(W^*)>e(Q_{q-1})$.
Lemma \ref{lemma3.03} gives that $\big||V_i|-\frac{n}{p}\big|\le\eta n$ for any $i\in [p]$.
Let $\widetilde{V}_i$ be the set of vertices of degree at least one in $G_i$ for each $i\in [p]$.
By Lemma \ref{lemma3.12}, $e(G_i)\leq \frac12\phi^2$,
  and so $|\widetilde{V}_i|\leq 2e(G_i)\leq \phi^2$.
Moreover, by Lemma \ref{lemma3.10}, $d_G(u)\geq (1-20\phi^5\eta )n$ for any $u\in W^*$.
Then,
$$\left|\bigcap_{u\in W^*}N_{V_i}(u)\right|
   \geq |V_i|-(q-1)\times20\phi^5\eta n\geq \frac{n}{p}-21\phi^6\eta n\geq |W^*|+|\widetilde{V}_i|+\frac{n}{2p},$$
where the last inequality holds as $\eta<\frac{1}{50\phi^7}$ and $n$ is sufficiently large.
It follows that there exists a subset $I_i^*\subseteq \left(\bigcap_{u\in W^*}N_{V_i}(u)\right)\setminus (W^*\cup \widetilde{V}_i)$ with $|I_i^*|=\lfloor\frac{n}{2p}\rfloor$.

Let $G'$ be the graph induced by $W^*\cup \big(\bigcup_{i=1}^{p}I_i^*\big)$.
By Claim \ref{claim3.4}, for any $i\in [p]$, every vertex in $I_i^*$ is adjacent to all vertices in
$W^*\cup \big(\bigcup_{j\in [p]\setminus\{i\}}I_j^*\big)$.
Moreover, since $e(W^*)>e(Q_{q-1})$, by the definition of $Q_{q-1}$ we can see that $G'$ contains a copy of $F^{p+1}$.
Consequently, $G$ also contains a copy of $F^{p+1}$ as $G'\subset G$,
which gives a contradiction.
The proof is complete.
\end{proof}

\begin{lem}\label{lemma3.14}
Assume without loss of generality that $n_1\geq n_2\geq \cdots \geq n_p$, where $n_i:=|V(G_i)|$ for each $i\in [p]$. Then $n_1-n_p\leq 1$.
\end{lem}

\begin{proof}
Suppose to the contrary that $n_1-n_p\geq 2$.
Choose an arbitrary integer $i\in [p]$.
By Claim \ref{claim3.4}, we may assume $x_u=x_i$ for each isolated vertex $u$ in $G_i$.
For any $v\in \widetilde{V}_i$,
by Lemma \ref{lemma3.09}, $(1-7\eta)x_{v^*}<x_v,x_i\leq x_{v^*}$.
This implies that
\begin{align*}
1-7\eta=\frac{(1-7\eta)x_{v^*}}{x_{v^*}}\leq \frac{x_{v}}{x_i}
            \leq \frac{x_{v^*}}{(1-7\eta)x_{v^*}}< 1+8\eta.
\end{align*}
By Lemma \ref{lemma3.12}, $|\widetilde{V}_i|\leq 2e(G_i)\leq {\phi^2}$ for any $i\in [p]$. Then
$$\big(|\widetilde{V}_i|-8\phi^2\eta\big)x_i< |\widetilde{V}_i|(1-7\eta)x_i \leq \sum_{v\in \widetilde{V}_i}x_v< |\widetilde{V}_i|(1+8\eta)x_i\leq \big(|\widetilde{V}_i|+8\phi^2\eta \big)x_i.$$
Then there exists a constant $\epsilon_i\in (-8\phi^2\eta,8\phi^2\eta)$ such that
\begin{align*}
\sum_{v\in \widetilde{V}_i}x_v= \big(|\widetilde{V}_i|+\epsilon_i\big)x_i.
\end{align*}
Hence,
\begin{align}\label{align0012}
\sum_{v\in V_i^*}x_v= (n_i-|\widetilde{V}_i|)x_i+\big(|\widetilde{V}_i|+\epsilon_i\big)x_i=(n_i+\epsilon_i\big)x_i.
\end{align}
By Claim \ref{claim3.4}, $\rho(G)x_i=\sum_{v\in V(G)\setminus V_i^*}x_v$. Consequently,
\begin{align*}
(\rho(G)+n_i+\epsilon_i)x_i=\sum_{v\in V(G)}x_v.
\end{align*}
Set $n_i^*=n_i+\varepsilon_i$ for each $i\in [p]$.
Recall that $\eta<\frac{1}{50\phi^7}$. Then,
\begin{align}\label{align0007}
n_1^*-n_p^*=(n_1-n_p)+(\epsilon_1-\epsilon_p)> 2-16\phi^2\eta>\frac53,
\end{align}
and
\begin{align}\label{align0008}
x_1=\frac{\sum_{v\in V(G)}x_v}{\rho(G)+n_1^*}~~\text{and}
~~x_p=\frac{\sum_{v\in V(G)}x_v}{\rho(G)+n_p^*}.
\end{align}
Choose an isolated vertex $u_0$ in $G_1$. So, $x_{u_0}=x_1$.
Let $G'$ be the graph obtained from $G$ by deleting edges
from $u_0$ to $V_p^*$
and adding all edges from $u_0$ to $V_1^*\setminus \{u_0\}$.
It is not hard to verify that $G'$ is still $F^{p+1}$-free.
However, in view of \eqref{align0007} and \eqref{align0008}, we get
\begin{center}
 $(n_1^*-1)x_1-n_p^*x_p=
\frac{(n_1^*-n_p^*-1)\rho(G)-n_p^*}{(\rho(G)+n_1^*)(\rho(G)+n_p^*)}\cdot\sum_{v\in V(G)}x_v>\frac{\frac23\rho(G)-\frac{n}{p}}{(\rho(G)+n_1^*)(\rho(G)+n_p^*)}\cdot\sum_{v\in V(G)}x_v>0,
$
\end{center}
where the last inequality holds as $p\geq 3$ and $\rho(G)>(\frac{p-1}{p}n-\frac{4p}{n})>\frac{3n}{2p}$.
This, together with \eqref{align0012}, leads to that
\begin{center}
 $\rho(G')-\rho(G)\ge
\sum\limits_{v\in V_1^*\setminus \{u_0\}}2x_{u_0}x_v-
   \sum\limits_{v\in V_p^*}2x_{u_0}x_v=2x_{u_0}\big((n_1^*-1)x_1-n_p^*x_p\big)>0,$
\end{center}
a contradiction.
Therefore, $n_1-n_p\leq 1$.
\end{proof}

Note that $|W^*|=q-1$ (see Lemma \ref{lemma3.10}) and $n_1-n_p\leq 1$ (see Lemma \ref{lemma3.14}).
Let $G_{cr}$ be the graph obtained from $G$ by deleting all edges in $E(W^*)\cup E(G_{in})$.
Then, $G_{cr}$ is a spanning subgraph of $H(n,p,q)$.
Let $G_{m}$ be the graph consisting of edges in $E(H(n,p,q))\setminus E(G_{cr})$
(it is possible that $G_m$ is empty).
Now we complete the proof of Theorem \ref{theorem2.2} for $q\geq 2$
( the proof for $q=1$ is similar and hence omitted here).
\begin{proof}
By the definition of $H(n,p,q,\lambda-1,\mathcal{B})$, we can see that
\begin{align}\label{align0009}
e(H(n,p,q,\lambda-1,\mathcal{B}))-e(G)=(e(D_{\lambda-1})+e(G_m)-e(G_{in}))+(e(Q_{q-1})-e(W^*)).
\end{align}
By Lemma \ref{lemma3.12}, $e(G_{in})\leq \sum_{i=1}^{p}e(G_i)\leq \frac12 p\phi^2$.
Combining Lemma \ref{lemma3.13}, we obtain
\begin{align}\label{align0010}
e(W^*)-e(Q_{q-1})+e(G_{in})\leq e(G_{in})\leq \frac12 p\phi^2.
\end{align}

\begin{claim}\label{claim3.5}
If $e(G)<e(H(n,p,q,\lambda-1,\mathcal{B}))$, then $\rho(G)<\rho(H(n,p,q,\lambda-1,\mathcal{B}))$.
\end{claim}

\begin{proof}
If $e(H(n,p,q,\lambda-1,\mathcal{B}))-e(G)\geq 1$,
then by \eqref{align0009} and \eqref{align0010},
\begin{align*}
(e(D_{\lambda-1})+e(G_m))(1-8\eta)^2
  &\geq  \big(e(W^*)-e(Q_{q-1})+e(G_{in})+1\big)(1-16\eta)\\
  &\geq \big(e(W^*)-e(Q_{q-1})+e(G_{in})\big)-16\eta\cdot \frac12 p\phi^2+(1-16\eta) \\
  &>  e(W^*)-e(Q_{q-1})+e(G_{in})+\frac{2}{5},
\end{align*}
where the last inequality holds as $\eta<\frac{1}{50\phi^7}$.
By Lemma \ref{lemma3.09}, we have
\begin{align*}
\sum_{uv\in E(D_{\lambda-1})\cup E(G_{m})}x_{u}x_{v}-\sum_{uv\in E(G_{in})}x_{u}x_{v}
  &\geq  \big((e(D_{\lambda-1})+e(G_{m}))(1-8\eta)^2-e(G_{in})\big)x_{v^*}^2\\
  &>\big( e(W^*)-e(Q_{q-1})+\frac{2}{5}\big)\frac{1}{4}x_{u^*}^2\\
  &> \big( e(W^*)-e(Q_{q-1})\big)\frac{1}{4}x_{u^*}^2 +\frac{1}{10}x_{u^*}^2.
\end{align*}
Note that $e(Q_{q-1})\leq \binom{q-1}{2}\leq \frac{\phi^2}{2}$. By Lemma \ref{lemma3.10}, we have
\begin{align*}
\sum_{uv\in E(Q_{q-1})}x_{u}x_{v}-\sum_{uv\in E(W^*)}x_{u}x_{v}
  &\geq  \big(e(Q_{q-1})(1-\eta)^2-e(W^*)\big)x_{u^*}^2\\
  &>\big(e(Q_{q-1})-e(W^*)-2\eta e(Q_{q-1}) \big)x_{u^*}^2\\
  &> \big(e(Q_{q-1})-e(W^*)\big)x_{u^*}^2-\frac{1}{20}x_{u^*}^2.
\end{align*}
Combining the above two inequalities, we have
\begin{align*}
    &\rho(H(n,p,q,\lambda-1,\mathcal{B}))-\rho(G)\\
    &\geq  {X^{\mathrm{T}}(A(H(n,p,q,\lambda-1,\mathcal{B}))-A(G))X}\\
    &\geq  2\left(\sum_{uv\in E(Q_{q-1})}x_{u}x_{v}-\sum_{uv\in E(W^*)}x_{u}x_{v}\right)
           +2\left(\sum_{uv\in E(D_{\lambda-1})\cup E(G_{m})}x_{u}x_{v}-\sum_{uv\in E(G_{in})}x_{u}x_{v}\right) \nonumber\\
  &\geq \big(e(Q_{q-1})-e(W^*)\big)\frac{3}{2}x_{u^*}^2+\frac{1}{10}x_{u^*}^2>0,
\end{align*}
where the last inequality holds as $e(Q_{q-1})\geq e(W^*)$ by Lemma \ref{lemma3.13}. The result follows.
\end{proof}

If ${\rm ex}(n,F^{p+1})<e(H(n,p,q,\lambda-1,\mathcal{B}))$,
then  $e(G)<e(H(n,p,q,\lambda-1,\mathcal{B}))$.
By Claim \ref{claim3.5}, we have $\rho(G)<\rho(H(n,p,q,\lambda-1,\mathcal{B}))$.
On the other hand, from the definition of $G$ and \eqref{align0001} we know that $\rho(H(n,p,q,0,\mathcal{B}))\leq \rho(G)$.
Therefore, Theorem \ref{theorem2.2} (i) holds.
Now consider the case ${\rm ex}(n,F^{p+1})=e(H(n,p,q,\lambda-1,\mathcal{B}))$.
By \eqref{align1001}, we have ${\rm EX}(n,F^{p+1})\subseteq \mathcal{H}(n,p,q,\lambda-1,\mathcal{B})$.
To prove Theorem \ref{theorem2.2} (ii),
it suffices to show that $G\in {\rm EX}(n,F^{p+1})$, that is, $e(G)=e(H^*)$ for any $H^*\in {\rm EX}(n,F^{p+1})$.
Suppose to the contrary that $e(G)<e(H^*)$.
By Claim \ref{claim3.5}, $\rho(G)<\rho(H^*)$,
which contradicts that $G$ is an extremal graph for ${\rm spex}(n,F^{p+1})$.
Therefore, $e(G)=e(H^*)$ and Theorem \ref{theorem2.2} (ii) holds.
\end{proof}

\begin{remark}
Combining Corollaries \ref{cor4.1}-\ref{cor4.5},
we can see that Question \ref{prob1.1} is valid for $F^{p+1}$ when $F$ is a path, a star, a cycle, or a complete graph.
Furthermore, combining Theorems \ref{theorem2.2} (ii), \ref{theorem3.1} (ii) and \ref{theorem3.2}, we can directly get Theorem \ref{thm1.001}. This means that Question \ref{prob1.1} is still valid for $F^{p+1}$ when $F$ is a non-bipartite graph, or a bipartite graph with $q<q(F)$.
However, it seems difficult to  verify the validity of Question \ref{prob1.1}  for $F^{p+1}$ when $F$ is a bipartite graph with $q=q(F)$, and so we leave this as a problem.
\end{remark}

\end{document}